\documentclass[12pt]{amsart}
\usepackage{amsfonts}
\usepackage{amsmath}
\usepackage{amssymb}

\usepackage[margin=1.2in]{geometry}
\newtheorem{theorem}{Theorem}[section]

\newtheorem{lemma}[theorem]{Lemma}

\theoremstyle{definition}

\newtheorem{remark}[theorem]{Remark}

\numberwithin{equation}{section}

\begin{document}
\title[Beurling numbers whose number of factors lies in a residue class]{Beurling numbers whose number of prime factors lies in a specified residue class}

\author[G. Debruyne]{Gregory Debruyne}
\thanks{G. Debruyne gratefully acknowledges support by the FWO through a postdoctoral fellowship. He thanks the Belgian American Educational Foundation for the opportunity to visit the University of Illinois at Urbana-Champaign for the academic year 2019-2020, during which this research was carried out.}
\address{G. Debruyne\\ Department of Mathematics: Analysis, Logic and Discrete Mathematics \\ Ghent University\\ Krijgslaan 281 Gebouw S8\\ B 9000 Gent\\ Belgium}
\email{gregory.debruyne@UGent.be}
\subjclass[2010]{11N25, 11N37, 11N80.}
\keywords{Beurling numbers; generalized primes; Hal\'asz theorem; residue class; number of prime factors}

\begin{abstract}
We find asymptotics for $S_{K,c}(x)$, the number of positive integers below $x$ whose number of prime factors is $c \; \mathrm{mod}\; K$. We study this question in the context of Beurling integers. 
\end{abstract}

\maketitle

\section{Introduction}

In classical prime number theory the relation $M(x) = o(x)$ for the the summatory function of the M\"obius function is well-known to be equivalent to the Prime Number Theorem (PNT). It says that asymptotically the integers\footnote{The squarefree integers to be precise. However, a simple elementary argument shows that the word ``squarefree" may be omitted.} with an even number of prime factors match those with an odd number of prime factors. Here we will study this question for arbitrary residue classes.\par

Let $K \geq 2$ be an integer and let $c$ be an integer lying in $[0,K)$. In this article we shall find an asymptotic formula for $S_{K,c}(x)$, the number of Beurling integers below $x$ whose number of prime factors is $c \;\mathrm{mod}\; K$. We shall show that in the classical integers case
\begin{equation} \label{eqpresska}
 S_{K,c}(x) \sim \frac{x}{K}.
\end{equation}
In the proof we aim to use as little information as possible on the integers and the primes. In fact, we shall only use a Chebyshev upper bound on the primes and density on the integers. We formalize this approach with the use of Beurling prime number systems.\par

A Beurling generalized prime number system is a sequence 
\begin{equation} \label{eqpresdefbeur}
 1 < p_{1} \leq p_{2} \leq \dots \rightarrow \infty.
\end{equation}
The generalized integers of the system are then formed by taking the multiplicative semigroup generated by the generalized primes and $1$. We usually write $p_{j}$ for generalized primes and $n_{j}$ for generalized integers and omit the subscripts if there is no risk of confusion. We have the counting functions $\pi(x)$ and $N(x)$ for the primes and integers respectively. The Riemann weighted prime counting function is defined as usual by
\begin{equation} \label{eqpresdefrie}
 \Pi(x) =  \sum_{k = 1}^{\infty} \frac{\pi(x^{1/k})}{k}.
\end{equation}
The zeta function
\begin{equation} \label{eqpreszeta}
 \zeta(s) = \sum_{n} n^{-s} = \int^{\infty}_{1^{-}} x^{-s} \mathrm{d}N(x),
\end{equation} 
defined on $\operatorname*{Re} s > 1$ (under the hypothesis $N(x) = O(x)$) is an indispensable tool for the study of these number systems via analytical methods.\par
Central questions in this theory involve how information on the integers has consequences for the primes and vice-versa. We refer to \cite{d-zbook} for a detailed account of Beurling generalized numbers. We mention that in this discussion we only concern ourselves with discrete Beurling numbers systems as defined above and not with more general definitions of Beurling number systems. The reason for doing so is that there is no straightforward generalization to the more general framework for the concepts we are working with here.\par

The main goal of this paper is to prove the following theorem.

\begin{theorem} \label{thprescheb}
Suppose a Beurling number system satisfies a Chebyshev upper bound $\pi(x) \ll x/\log x$ and has a positive density $N(x) \sim ax$, then
\begin{equation} \label{eqpresrb}
 S_{K,c}(x) \sim \frac{ax}{K}.
\end{equation}
\end{theorem}
A crucial ingredient of the proof is a recently established version \cite{d-m-v} of Hal\'asz's theorem for Beurling numbers. This topic was introduced by Zhang in \cite{zhang1987} whose results he later improved upon in \cite{zhang2018} based on ideas of \cite{d-d-v}. The paper \cite{d-m-v} refines these ideas further and contains the best results currently available.

\section{Proof of Theorem \ref{thprescheb}}

Let $q$ be an integer in $[0,K)$. We define $f_{q}(n)$ depending on the number of prime factors of $n$, namely,
\begin{equation} \label{eqpresdeffq}
 f_{q}(n) = e^{2\pi iql/K} \text{ if the number of prime factors of } n \text{ is } l \;\mathrm{mod}\; K.
\end{equation}
Then
\begin{equation} \label{eqpresorth}
 S_{K,c}(x) = \sum_{n \leq x} \frac{1}{K} \sum_{q = 0}^{K-1} e^{-2\pi iqc/K} f_{q}(n).
\end{equation}
We shall show that the summatory function $F_{q}(x) := \sum_{n \leq x} f_{q}(n) = o(x)$, unless $q = 0$, in which case $F_{0}(x) = N(x) \sim ax$. The orthogonality relation \eqref{eqpresorth} would then complete the proof. Let us now calculate the generating function $\hat{F_{q}}$ of the multiplicative function $f_{q}$, for $\operatorname*{Re} s > 1$, 
\begin{equation} \label{eqpreseuler}
 \hat{F_{q}}(s) = \sum_{n} \frac{f_{q}(n)}{n^{s}} = \prod_{p}\left(1 + e^{2\pi iq/K}p^{-s} + e^{4\pi iq/K}p^{-2s} + \dots \right) = \prod_{p} \frac{1}{1- e^{2\pi iq/K}p^{-s}}.
\end{equation}
Taking logarithms, one obtains
\begin{equation} \label{eqpreslogeuler}
 \log \hat{F_{q}}(s) = \sum_{p}\sum_{k = 1}^{\infty} \frac{e^{2\pi ikq/K}p^{-ks}}{k}.
\end{equation}
Therefore, $\log \hat{F_{q}}(s)$ is the Mellin-Stieltjes transform of the measure $h_{q}\mathrm{d}\Pi$, where $h_{q}$ is a function supported on prime powers and satisfying $h_{q}(p^{k})= e^{2\pi ikq/K}$. It follows that $\hat{F_{q}}(s)$ is the Mellin-Stieltjes tranform of $\exp^{\ast}(h_{q}\mathrm{d\Pi})$, where $\exp^{\ast}$ is the exponential taken with respect to the multiplicative convolution of measures (see e.g. \cite{b-d} or \cite{d-zbook}). Hence, $\mathrm{d}F_{q} = \exp^{\ast}(h_{q}\mathrm{d\Pi})$, as the Mellin-Stieltjes transform is an injective operation.\par
We now wish to apply Theorem 2.3 from \cite{d-m-v}. We split $h_{q}$ into $g_{1} + g_{2}$ as $e^{2\pi iq/K} + (h_{q}-e^{2\pi iq/K})$. The assumptions for $g_{1}$ and $g_{2}$ from Theorem 2.3 are fulfilled because of the Chebyshev bound and the fact that $(h_{q}-e^{2\pi iq/K}) \mathrm{d}\Pi$ is supported only on (higher order) prime powers. It now remains only to verify that the Mellin-Stieltjes transform of $\exp^{\ast}(h_{q}\mathrm{d\Pi})=o(1/(\sigma - 1))$ uniformly for $t$ on compacts as $\sigma \rightarrow 1^{+}$. It suffices to show the Mellin-Stieltjes transform of $\exp^{\ast}(e^{2\pi iq/K}\mathrm{d\Pi})$ admits this bound, for, as shown in the proof of Theorem 2.3 in \cite{d-m-v}, the part coming from $g_{2}$ is harmless. We are thus left to show that, uniformly for $t$ on compacts,
\begin{equation} \label{eqpresmodzeta}
 \exp\left(e^{2\pi iq/K} \log \zeta(\sigma + it)\right) = o\left( \frac{1}{\sigma - 1}\right),
\end{equation}
which we do in the following lemma.
\begin{lemma} If $N$ has positive density $a$, then, for $q \in \{1, \dots, K-1\}$, the relation \eqref{eqpresmodzeta} holds uniformly for $t$ on compacts.
\end{lemma}
\begin{proof} The proof of this lemma goes along similar lines as the one of \cite[Lemma 3.6]{d-d-v}. First we show \eqref{eqpresmodzeta} pointwise. If $t = 0$, then $|\exp\left(e^{2\pi iq/K} \log \zeta(\sigma)\right)| = |\zeta(\sigma)|^{\cos(2\pi q/K)} = o(1/(\sigma - 1))$ as $\zeta(\sigma) \sim  a/(\sigma-1)$, which is implied by density. For $t \neq 0$, we shall employ the trigonometrical inequality
\begin{equation} \label{eqpresgonin}
  M-1 - M \cos(x) + \cos(Kx) \geq 0,
\end{equation}
which is valid for all real $x$ and positive integers $K$ as long as $M \geq K^2$. Indeed, let $f(x)$ be the left-hand side of \eqref{eqpresgonin}. As $f$ is $2\pi$-periodic it suffices to show \eqref{eqpresgonin} for $|x| \leq \pi$, or even only for $|x| \leq \sqrt{6}/\sqrt{M}$, as the only potential violations can occur when $\cos(x) \geq (M-2)/M$, and $\cos(x) \leq 1- x^{2}/3$ for $|x| \leq 2$, say. As $f(0) = 0$ and $f'(0) = 0$, \eqref{eqpresgonin} will follow if $f''(x) \geq 0$ for $|x| \leq \sqrt{6}/\sqrt{M}$, allowing a possible equality only when $x = 0$. As $f''(x) = M\cos x - K^{2} \cos(Kx)$, positivity follows as $\cos(Kx) \leq \cos(x)$ for $|x| \leq \sqrt{6}/\sqrt{M}$, with equality only if $x = 0$. (Note that $|Kx| \leq \sqrt{6}$ in this region, so $Kx$ cannot wander fully around the circle.) This completes the verification of \eqref{eqpresgonin}.\par
We employ the trigonometrical inequality \eqref{eqpresgonin} to find
\begin{equation}
 \int^{\infty}_{1} x^{-\sigma}\left(M - 1 - M\cos(t\log x -2\pi q/K) + \cos(Kt\log x)\right)\mathrm{d}\Pi(x) \geq 0,
 \end{equation}
or, after exponentiation
\begin{equation}
 \zeta(\sigma)^{M-1} \exp(-M \operatorname*{Re}(e^{2\pi i q/K} \log \zeta(\sigma + it)))  |\zeta(\sigma + iKt)| \geq 1.
\end{equation}
The pointwise result now follows as density implies $\zeta(\sigma) \sim a/(\sigma - 1)$ and $\zeta(\sigma + iKt) = o_{t}(1/(\sigma - 1))$.\par Now, \eqref{eqpresmodzeta} is equivalent to
\begin{equation} \label{eqpresdini}
 \exp\left( - \int^{\infty}_{1} x^{-\sigma}(1 - \cos(t\log x - 2\pi q /K)) \mathrm{d}\Pi(x)\right) = o(1).
\end{equation}
Our previous considerations imply that the left-hand side of \eqref{eqpresdini}, a net of continuous functions, monotone in the variable $\sigma$ as $\sigma \downarrow 1$, tends pointwise to $0$. Therefore, Dini's theorem asserts that the convergence of \eqref{eqpresdini} must happen uniformly for $t$ on compact sets.
\end{proof}


\begin{remark} We also observe that Theorem \ref{thprescheb} remains true if we change the definition of $S_{K,c}(x)$ to the number of integers below $x$ for which the number of \emph{distinct} prime factors is $c \;\mathrm{mod}\; K$. The proof is similar to the one given above. In the definition \eqref{eqpresdeffq} of $f_{q}$, one simply adds the word ``distinct" at the appropriate place. The orthogonality relation  \eqref{eqpresorth} remains valid. The generating function of $f_{q}$ becomes
\begin{equation} \label{eqpresaltgen}
 \hat{F_{q}}(s) = \prod_{p} \left(1 + e^{2\pi i q/K} p^{-s} + e^{2\pi i q/K} p^{-2s} + \dots \right)
\end{equation}
Proceeding in the same fashion as above, one obtains after some computations 
\begin{equation} \label{eqpresaltlog}
 \log \hat{F_{q}}(s) = \sum_{p}\sum_{k = 1}^{\infty} \frac{(1 - (1-e^{2\pi iq/K })^{k})p^{-ks}}{k}.
\end{equation}
Now however, we will split the measure $h_{q} \mathrm{d}\Pi$ into three pieces. We set $g_{1}(p^{k}) = e^{2\pi i q/K}$,
\begin{equation} \label{eqpresg2} 
 g_{2}(p^{k}) = 
\begin{cases}
 1-e^{2\pi iq/K} - (1-e^{2\pi iq/K})^{k} & \text{if } p > 2,\\
  -e^{2\pi iq/K} & \text{if } p \leq 2,
\end{cases}
\end{equation}
and $g_{3}(p^{k}) = 1-(1-e^{2\pi iq/K})^{k}$ if $p \leq 2$ and $0$ otherwise. Now, as $|g_{2}(p^{k})| \leq 2^{k}$ for $p > 2$ and is bounded for $p \leq 2$, the hypothesis $\int^{\infty}_{1} |g_{2}(u)|u^{-1}\mathrm{d}\Pi(u) < \infty$ from \cite[Th. 2.3]{d-m-v} is fulfilled and one can proceed exactly as above to obtain $\int^{x}_{1^{-}}\exp^{\ast}((g_{1}+g_{2}) \mathrm{d}\Pi) = o(x)$ for $q \neq 0$. Now the Mellin-Stieltjes transform of $\exp^{\ast}(g_{3} \mathrm{d}\Pi)$ is 
\begin{equation} \label{eqpreseulerg3}
 \prod_{p \leq 2} \left(1+  \sum_{k = 1}^{\infty} \frac{f_{q}(p^{k})}{p^{ks}}\right),
\end{equation}
and is absolutely convergent for $\sigma > 0$ by virtue of $|f_{q}| \leq 1$. Applying \cite[Lemma 3.4(i)]{d-m-v} now gives
\begin{equation} \label{eqpresconv}
F_{q}(x) = \int^{x}_{1^{-}} \exp^{\ast}((g_{1}+g_{2}) \mathrm{d}\Pi) \ast \exp^{\ast}(g_{3} \mathrm{d}\Pi) = o(x) 
\end{equation}
and therefore also $S_{K,c}(x) \sim ax/K$.
\end{remark}

\begin{remark} The density condition in Theorem \ref{thprescheb} can be replaced with the apparently weaker hypothesis of logarithmic density, that is, the following limit
\begin{equation}\label{eqpreslogdens}
 \lim_{x \rightarrow \infty} \int^{x}_{1^{-}} \frac{\mathrm{d}N(u)}{u}
\end{equation}
exists and is positive. It can be easily seen via partial integration that density implies logarithmic density unconditionally. On the other hand it follows from Corollary 2.2 from \cite{d-m-v} that logarithmic density implies density under a Chebyshev upper bound condition, as logarithmic density implies equation (2.5) of that paper.
\end{remark}

\end{document}